\documentclass[10pt]{article}
\usepackage[a4paper]{anysize}\marginsize{3.5cm}{3.5cm}{1.3cm}{2cm}
\pdfpagewidth=\paperwidth \pdfpageheight=\paperheight
\usepackage{amsfonts,amssymb,amsthm,amsmath,eucal}
\usepackage{pgf}
\usepackage{bbm,array}
\usepackage{tikz} 
\usepackage{float}
\usepackage{subfigure}
\usepackage{caption}
\usetikzlibrary{arrows}
\usepackage{mathbbol}
\usepackage{url}

\usepackage{hyperref}
\hypersetup{
    colorlinks,
    citecolor=black,
    filecolor=black,
    linkcolor=black,
    urlcolor=black
}

\theoremstyle{plain}
\newtheorem{thm}{Theorem}[section]
\newtheorem{theorem}[thm]{Theorem}
\newtheorem*{theoremA}{Theorem A}

\theoremstyle{definition}
\newtheorem{definition}[thm]{Definition}
\newtheorem{remark}[thm]{Remark}

\newtheorem{question}[thm]{Question}

\newcommand\be{\begin{eqnarray*}}
\newcommand\ee{\end{eqnarray*}}

\newcommand\C{\mathbb C}

\newcommand\K{\mathbb K}

\newcommand\cala{{\mathcal A}}
\newcommand{\calaif}{\mathcal{IF}}
\newcommand{\calarf}{\mathcal{RF}}
\newcommand{\PP}{\mathbb{P}}

\newcommand\newop[2]{\def#1{\mathop{#2}\nolimits}}
\newop\mod{mod}
\newop\Der{Der}
\newop\pdeg{pdeg}
\newop\log{log}
\newop\ord{ord}
\newop\Gal{Gal}
\newop\SL{SL}
\newop\GL{GL}
\newop\Bl{Bl}
\newop\mult{mult}
\newop\mass{mass}
\newop\div{div}
\newop\codim{codim}
\newop\Jac{Jac}
\newop\sing{sing}
\newop\vdim{vdim}
\newop\edim{edim}
\newop\Ass{Ass}
\newop\size{size}
\newop\areg{areg}
\newop\asreg{asreg}
\newop\satdeg{satdeg}
\newop\supp{supp}
\newop\gin{gin}
\newop\ini{in}
\newop\vol{vol}
\newop\sat{sat}
\newop\length{length}
\newop\depth{depth}
\newop\characteristic{char}

\usepackage{tcolorbox}
\definecolor{mycolor}{rgb}{0.122, 0.435, 0.698}
\newtcolorbox{mybox}{colback=blue!5!white,colframe=mycolor}

\def\keywordname{{\bfseries Keywords}}%
\def\keywords#1{\par\addvspace\medskipamount{\rightskip=0pt plus1cm
\def\and{\ifhmode\unskip\nobreak\fi\ $\cdot$
}\noindent\keywordname\enspace\ignorespaces#1\par}}
\def\subclassname{{\bfseries Mathematics Subject Classification
(2000)}\enspace}
\def\subclass#1{\par\addvspace\medskipamount{\rightskip=0pt plus1cm
\def\and{\ifhmode\unskip\nobreak\fi\ $\cdot$
}\noindent\subclassname\ignorespaces#1\par}}

\begin{document}

\author{Marek Janasz}
\title{On the containment problem and sporadic simplicial line arrangements}
\date{\today}
\maketitle
\thispagestyle{empty}

\begin{abstract}

In the paper we present two examples of inductively free sporadic simplicial arrangements of $31$ lines that are non-isomorphic, which allow us to  answer negatively questions on the containment problem recently formulated by Drabkin and Seceleanu.
 \subclass{14N20, 13A15, 52A20}
\end{abstract}
\section{Introduction}
Study of the relation between symbolic and 
ordinary power of homogeneous ideals in the polynomial ring over a given field $\mathbb{K}$ has a long history and is derived from many different problems in mathematics. In $1995$, Eisenbud and Mazur in \cite{EisenbudMazur}, 
referring to the proof of Fermat's Last Theorem, were investigating 
the so-called ''fitting ideals'' and some symbolic powers of certain associated ideals. What they 
proved, among other things, is that $I^{(2)} \subset 
\mathfrak{m}I$ in the case of perfect ideals of 
codimension $2$. They also showed that this kind of the containment
holds for several other classes of 
ideals. In $2013$, Harbourne and Huneke in \cite{HaHu} 
proposed a certain generalization and began to study the 
relation $I^{(m)} \subset \mathfrak{m}^kI^r$, and their work was continued in 
\cite{BisGrifHaNg2021,CopSusHu2017}.

Another way, which in fact ends in the investigations on
symbolic and ordinary powers of ideals, were the articles by Skoda \cite{Skoda} and Waldschmidt 
\cite{Wald}. They focused on some estimates of the 
degree of hypersurfaces in $\PP^N_{\mathbb{K}}$ passing through fixed points with prescribed multiplicities. A paper by 
Chudnovski \cite{Chud} fits into these considerations. 
Using some complex analysis tools, Chudnovski 
improved the results of Skoda and Waldschmidt in $\PP^2$ 
and he formulated a still open conjecture for the case of $\PP^N$. The~generalization of this conjecture was also given by Demailly in \cite{Dema1982}, and the combination of these conjectures is the subject of intense research 
\cite{BisGrifHaNg2021,BocHarb2010,Dum2015,DumTutaj2017,
FouliXie2018,MalSzpSzem2018}.

Using the Nagata-Zariski theorem makes it possible to relate geometric questions to algebra. Therefore, the study of containment relations between ordinary and symbolic powers of homogeneous ideals of points is a connection between conjectures formulated by Chudnovsky and Demailly. This perspective led to an increased interest in the so-called containment problem, i.e., the determination of the exponents $(m,r)$ for which the $m$-th symbolic power of a homogeneous ideal is contained in the $r$-th ordinary power of that ideal. An initial work on this topic was begun by Hochster in $1973$ \cite{Hoch73}, but the groundbreaking result was published only in $2001$ by Ein, Lazarsfeld, and Smith \cite{ELS}, where they gave a lower bound on the exponent of the symbolic power, which depends on the dimension of the space $\PP^N$. Since then, the cases unsolved by the aforementioned theorem have become the subject of intensive study, in particular the smallest case from the perspective of the magnitude of the powers, namely the containment $I^{(3)} \subset I^2$ for the ideals of reduced points in $\PP^2$. While at the beginning the researchers tried to prove that this particular containment holds for all homogeneous ideals, after the paper \cite{DumSzGas}, where the first counterexample defined over $\mathbb{C}$ has been presented, a lot of counterexamples defined over different fields have been published (see \cite{CzLamp2016,MalSzp2017}). Despite a growing number of counterexamples, the true nature of the relation between $I^{(3)}$ and $I^2$ is still unknown. In \cite{DraSec}, Drabkin and Seceleanu study reflection arrangements given by (irreducible) complex pseudoreflection groups. As a result, they give a complete description of the relation between the third symbolic power and the second ordinary power of radical ideal $J(\cala)$, which defines the singular locus of the complex reflection arrangement $\cala$. The work on this problem motivates them to state some open questions, among which they ask (\cite[Question 6.7.-6.8.]{DraSec}): \textit{Are the containments $(J(\cala))^{(2r-1)}\subseteq (J(\cala))^r$ always satisfied for any $r \geq 2$ and any hyperplane arrangement that is inductively/recursively free?} 

In the present paper we give a negative answer to these questions, namely we prove the following.

\begin{theoremA}
There are two non-isomorphic inductively free simplicial arrangements consisting of $31$ lines such that they have the same weak combinatorics, and having the property that for one arrangement the containment $(J(\cala))^{(3)}\subseteq (J(\cala))^2$ holds, but does not for the other. 
\end{theoremA}

The structure of the paper is as follows. In Section \ref{sec:prem}, we recall some basic definitions and tools that we will use in the rest of this paper concerning line arrangements and symbolic powers of homogeneous ideals. In Section \ref{sec:A12k7} we give very detailed information about a family of line arrangement known as $\cala(12k+7)$, giving line equations and proving that some line arrangements from this family are inductively free. This result is used in Section \ref{sec:main}, where we prove Main Theorem of this paper. At the end, we provide our \verb}SINGULAR} code to let interested readers check the containment between $(J(\cala))^{(3)}$ and $(J(\cala))^{2}$.

\section{Preliminaries}
\label{sec:prem}
In this section we recall all necessary definitions regarding hyperplane arrangements that we will exploit in the paper. For more information regarding this subject, please consult \cite{dimca17,OrlTer92}.

Let $\K$ be a field of characteristic zero and let $V$ be a fixed vector space of dimension $\ell$ over $\K$.  Let $\{x_1,\ldots, x_{\ell}\}$ be the dual basis of $V^{*}$ associated with $V$, then the symmetric algebra $S(V^{*})$ is isomorphic to the ring of polynomials $S=\K[x_1,\ldots, x_\ell].$ 

A pair $(\cala,V)$ is called an $\ell$-arrangement of hyperplanes, i.e., this is an arrangement of dimension $(\ell - 1)$ linear subspaces in $V$. The symbol $\Phi_{\ell}$ denotes the empty $\ell$-arrangement. If the dimension is clear from the context, we use the name arrangement for short. Each hyperplane $H\in\cala$ is the kernel (up to a constant) of a linear form $l_H\in V^*$. The product of all linear forms
\[Q(\cala)=\prod_{H\in\cala}l_H\]
is called the defining polynomial of $\cala$. In the case of empty arrangements, we put $Q(\Phi_l)=1$.

By $L(\cala)$ we denote the intersection lattice of $\cala$, 
i.e., the set of all non-empty intersections of hyperplanes $H_i$ in $\cala$. Taking any $X\in L(\cala)$, a subarrangement $\cala_X$ of $\cala$ 
is called {\it localization} and it is defined as 
\[\cala_X=\{H\in \cala\:|\: X\subseteq H\}.\]
For a chosen $X$, we define a subarrangement of $\cala$ by
\[\cala^X=\{X\cap H\:: X \not\subseteq H\:  
\text{and}\: X\cap H \not = \emptyset \},\]
which we call the {\it restriction} of $\cala$ to $X$.
\begin{definition}
A  simplicial arrangement is a finite set $\cala = \{H_1, \ldots, H_{n}\}$ of (central) hyperplanes in $\mathbb{R}^{\ell}$ such that all connected components of the complement \[M(\mathcal{A}):=\mathbb{R}^\ell\setminus\bigcup_{H\in\cala}H\] are simplicial cones.
\end{definition}
Denote by $\Der_{\K}(S)$ the set of all $\K$-linear maps (derivations) 
$\theta: S \longrightarrow S$ such that for all $f,g \in S$ one has
\[\theta(fg)=f \theta(g)+g\theta(f). \]   It is known that the set $\left\{\tfrac{\partial}{\partial x_i}\right\}_{i=1}^{\ell}$ forms a (canonical) basis for $\Der_\K(S)$, i.e.,
$$\Der_\K(S) = \bigoplus_{i=1}^{\ell} S \cdot \tfrac{\partial}{\partial x_i}.$$
Any homogeneous element $0 \not = \theta \in \Der_\K(S)$ can be expressed as $\theta = \sum_{i=1}^{\ell}g_i \cdot \tfrac{\partial}{\partial x_i}$, where $g_i\in S$ are homogeneous polynomials of degree $d$. For such $\theta$ we denote by $\pdeg\theta = d$ its polynomial degree. 

For any $f\in S$ being homogeneous, we define an $S$-submodule of 
$\Der_{\K}(S)$ as \[D(f)=\{\theta\in\Der_{\K}(S):\theta(f)\in f \cdot S\}.\] 
In the case of arrangement $\cala$, we use the notation $D(\cala)$ instead of $D(Q(\cala))$.
\begin{definition}
If $D(\cala)$ is a free $S$-module, then we say that $\cala$ is a \emph{free arrangement}.
\end{definition}
Let $\cala$ be a free arrangement for which  $\{\theta_1, \ldots, \theta_\ell\}$ is a homogeneous basis of $D(\cala)$. We say that the set of integers $\exp(\cala)=\{\pdeg\theta_1,\ldots,\pdeg\theta_\ell\}$ is the set of the exponents of $\cala$.

If we denote by $\theta_E \in \Der_{\K} (S)$ the Euler derivation, then we have the decomposition of $D(\cala)$, namely
$$D(\cala) = S \cdot \theta_E \oplus D_0(\cala).$$


 
Now for an arrangement $\cala$ and fixed $H\in \cala$, it is convenient
to study triples of arrangements $(\cala,\cala^{'},\cala^{''})$ of 
arrangements, where $\cala^{'}=\cala\setminus\{H\}$ and 
$\cala^{''}=\cala^{H}.$ The next theorem is very useful in all our considerations.

\begin{thm}\label{AD1}(Addition-Deletion, see \cite{OrlTer92})
Suppose $\mathcal{A} \not = \Phi_\ell$ . 
Let $(\cala, \cala^{'}, \cala^{''})$ be
a triple. Any two of the following 
statements imply the third:
\begin{align*}
\mathcal{A} \text{ is free with } {\rm exp}(\mathcal{A}) &= 
\{ b_1 , \ldots ,b_{\ell - 1} , 
b_{\ell}\},\\
\mathcal{A}^{'} \text{ is free with } {\rm exp}( 
\mathcal{A}^{'}) & = \{b_1, \ldots , b_{\ell-
1}, b_{\ell} - 1\},\\
\mathcal{A}^{''} \text{ is free with } {\rm exp} 
(\mathcal{A}^{''}) & = \{b_1, \ldots. ,b_{\ell - 
1} \}.
\end{align*}
\end{thm}

In this paper we deal with line 
arrangements $\cala$ defined over the complex numbers, therefore we will use the following reformulation of Theorem \ref{AD1}.
\begin{thm}\label{DiMar}
Let $\cala$ be a line 
 arrangement in $\PP^2_{\C}$ and $H \in  \cala$. Let $\cala' 
 :=\cala\setminus\{ H \}$. If the 
 following conditions hold:
	\begin{enumerate}
		\item[(1)] $\cala'$ is free and has the exponents ${\rm exp}(\mathcal{A}') = \{1, a, b\}$,
		\item[(2)]$|Sing(\cala)\cap H| = b + 1$ (or $a + 1$, respectively),
	\end{enumerate}
	then $\cala$ is free with the exponents ${\rm exp}(\cala)=\{1, a + 1, b\}$ (or ${\rm exp}(\cala) = \{1, a, b + 1\}$, respectively).
\end{thm}
In the sequel, we focus on the following definition.

\begin{definition}(\cite[Definition 4.53]{OrlTer92}).\label{def:IF}
	The class $\calaif$ of inductively free arrangements is the smallest class of arrangements which satisfies both conditions:
\begin{enumerate}\itemsep0pt \parskip0pt \parsep0pt
		\item[(1)] $\Phi_{\ell}\in\calaif$  for $\ell\geq 0$,
		\item[(2)] if there exists $H\in\cala$ such that $\cala''\in\calaif $, 
		$\cala'\in\calaif$, and ${\rm exp}(\cala'')\subset {\rm exp}(\cala')$,
		then $\cala\in\calaif$.
	\end{enumerate}
\end{definition}

\section{Examples of inductively free arrangements}

\label{sec:A12k7}
The main object of our considerations is a special family of line arrangements, denoted in the literature by $\cala(12k+7)$. This infinite family was originally described in the paper by Gr{\"u}nbaum \cite{Gru98}. Here, we recall this construction and its basic properties. 

For fixed $k$, each element of the family consists of exactly $12k+7$ lines, including the line at infinity $z=0$. The equations of these lines are given explicitly in Table \ref{tb:A12k7}.

{\renewcommand{\arraystretch}{1.0}\begin{table}[h]
		\begin{center}
			\begin{tabular}{ll} \hline
				\multicolumn{2}{c}{$\cala(12k+7)$} \\
				\hline\hline
				
				$2x-eiz$, &   \\
				
				$ x-ey+iez$, &  for $i \in \{-(k+1),-k,\ldots,-1,0,1,\ldots,k,k+1\}$\\
				
				$ x+ey-iez$, &     \\\hline
				
				$2y-jz$, &  \\

 				$ex-y+jz$, & for $j \in \{-(k-1),-(k-2),\ldots,-1,0,1,\ldots,k-2,k-1\}$\\
				
				$ex+y-jz$, &  \\\hline
				
				$z$\\\hline\hline
			\end{tabular}
\end{center}
		\caption{$\:$ Equations of lines of $\cala(12k+7)$.}
		\label{tb:A12k7}
	\end{table}
}


The arrangements $\cala(19)$ and $\cala(31)$ are exactly the sporadic simplicial arrangements $A(19,1)$ and $A(31,2)$ listed in Gr{\"u}nbaum's catalogue \cite{Gru}. From \cite{DiMar}, we know that these are free arrangements with the exponents ${\rm exp}(\cala(19))= \{1, 7, 11\}$ and ${\rm exp}(\cala(31)) = \{1, 13, 17\}$. 
We start with the following combinatorial observation that is crucial for our further considerations. According to our best knowledge, this observation is not known in the literature.
\begin{theorem}\label{thFree1} 
The arrangements of line 
$\cala(19,1)$ and $\cala(31,2)$ are 
inductively free.
\end{theorem}

\begin{proof}
We will divide our proof of this theorem into two steps. In the first step, we will show that the arrangement $\cala(19,1)$
is inductively free. For this purpose, we present Table \ref{tab:19Free} below, where we deliver the sequences of the exponents for 
$\cala^{'}$, the equation of each line that we 
add to the arrangement starting from $\Phi_3$, and then the exponents
of ${\rm exp}(\cala^{''})$. Each subsequent row of the table 
allows us to verify the conditions contained in Theorem \ref{DiMar}.
{\renewcommand{\arraystretch}{1.0}\begin{table}[h]
\begin{center}
\begin{tabular}{lll|lll} \\
\hline 
$\exp\:\cala^{'}$ & $\ell_i$ &  $\exp\:\cala^{''}$ &
$\exp\:\cala^{'}$ & $\ell_i$ &  $\exp\:\cala^{''}$
\\ \hline 
$\{0,0,0\}$ & $\Phi_3$ & $\{0,0\}$ &  $\{1,4,5\}$ & $\ell_{10}: x+ey+ez$ & $\{1,5\}$ \\
$\{0,0,1\}$ & $\ell_1: z $ & $\{0,1\}$ & $\{1,5,5\}$ & $\ell_{11}: x-ey-ez$ & $\{1,5\}$   \\
$\{0,1,1\}$ & $\ell_2: ex+y$ & $\{1,1\}$ &$\{1,5,6\}$ & $\ell_{12}: x+ey-ez$ & $\{1,5\}$ \\
$\{1,1,1\}$ & $\ell_3: ex-y$ & $\{1,1\}$ & $\{1,5,7\}$ & $\ell_{13}: x-ey+ez$ & $\{1,7\}$ \\
$\{1,1,2\}$ & $\ell_4: y$ & $\{1,1\}$ & $\{1,6,7\}$ & $\ell_{14}: 2x-2ez$ & $\{1,7\}$ \\
$\{1,1,3\}$ & $\ell_5: x$ & $\{1,1\}$ & $\{1,7,7\}$ & $\ell_{15}: x+ez$ & $\{1,7\}$\\
$\{1,1,4\}$ & $\ell_6: x+ey$ & $\{1,1\}$ & $\{1,7,8\}$ & $\ell_{16}: x+ey+2ez$ & $\{1,7\}$\\
$\{1,1,5\}$  & $\ell_7: x-ey$ & $\{1,5\}$ & $\{1,7,9\}$ & $\ell_{17}: x-ey+2ez$ & $\{1,7\}$ \\
$\{1,2,5\}$ & $\ell_8: 2x-ez$ & $\{1,5\}$ & $\{1,7,10\}$ & $\ell_{18}: x-ey-2ez$ & $\{1,7\}$\\
$\{1,3,5\}$  & $\ell_9: 2x+ez$ & $\{1,5\}$ & $\{1,7,11\}$ & $\ell_{19}: x+ey-2ez$ & $\{1,7\}$\\ \hline
\end{tabular}
\caption{. List of
 the exponents for arrangements building $\cala(19)$, 
 where $e=\sqrt{3}$. }\label{tab:19Free}
\end{center}
\end{table}
}

Based on the last row in Table \ref{tab:19Free}, we see 
that the arrangement $\cala(19)$ is a free with the 
exponents ${\rm exp}(\mathcal{A}(19))= \{1, 7, 11\}$.

For the second part of the proof, we apply Theorem
\Ref{DiMar} to arrangement $\mathcal{A}(19)$ by adding suitably chosen lines that are indicated in Table \Ref{tab:Ln} below.
\newpage
\begin{table}[!ht]
\begin{center} 
\begin{tabular}[b]{cl|cl}
\hline
$\ell_{19+i}$ & Equations of lines & $\ell_{19+i}$  &  
Equations of lines \\
\hline
$\ell_{20}:$ & $ex+y-z$ &  
$\ell_{26}:$ & $x+ey+3ez$ \\     
$\ell_{21}:$ & $ex+y+z$ &  
$\ell_{27}:$ & $x+ey-3ez $ \\
$\ell_{22}:$ & $ex-y+z$ & 
$\ell_{28}:$ & $x-ey+3ez $\\       
$\ell_{23}:$ & $ex-y-z$ & 
$\ell_{29}:$ &$x-ey-3ez $\\       
$\ell_{24}:$ & $2y-z$ & 
$\ell_{30}:$ & $2x+3ez$\\
$\ell_{25}:$ & $2y+z$ 
& $\ell_{31}:$ & $2x-3ez$\\
\hline
\end{tabular}
\caption{. Equations of lines 
$\ell_{19+i} $ with $i \in \{1, ..., 12\}$.}
\label{tab:Ln}
\end{center}
\end{table}

Attaching successively lines $\ell_{20},...,\ell_{31}$ to $\mathcal{A}(19)$, and using each time Theorem 
\ref{DiMar}, we conclude that the obtained 
arrangements are free. All the details describing our procedure are presented in the following diagram below:

{ \footnotesize
\[\begin{array}{c}\label{schemat}
\cala(19) \\
{\rm exp}(\cala) = \{1,6,11\}
\end{array}
\longrightarrow
\begin{array}{c}
\cala(19)\cup\{\ell_{20}\} \\
{\rm exp}(\cala) = \{1,8,11\}
\end{array}
\longrightarrow\ldots\]
\[
\longrightarrow
\begin{array}{c}
\cala(19)\cup\{\ell_{20},\ldots, \ell_{24}\} \\
{\rm exp}(\cala) = \{1,11,11\}
\end{array}
\longrightarrow
\begin{array}{c}
\cala(19)\cup\{\ell_{20},\ldots, \ell_{25}\} \\
{\rm exp}(\cala) = \{1,11,12\}
\end{array}
\longrightarrow
\]
\[
\longrightarrow
\begin{array}{c}
\cala(19)\cup\{\ell_{20},\ldots, \ell_{26}\} \\
{\rm exp}(\cala) = \{1,11,13\}
\end{array}
\longrightarrow
\begin{array}{c}
\cala(19)\cup\{\ell_{20},\ldots, \ell_{27}\} \\
{\rm exp}(\cala) = \{1,12,13\}
\end{array}
\]

\[\longrightarrow
\begin{array}{c}
\cala(19)\cup\{\ell_{20},\ldots,\ell_{28}\}\\
{\rm exp}(\cala) = \{1,13,13\}
\end{array}
\longrightarrow 
\begin{array}{c}
\cala(19)\cup\{\ell_{20},\ldots,\ell_{29}\} \\
{\rm exp}(\cala) = \{1,13,14\}
\end{array}
\longrightarrow\]
\[
\ldots
\longrightarrow
\begin{array}{c}
\cala(19)\cup\{\ell_{20},\ldots, \ell_{31}\}\\
{\rm exp}(\cala) = \{1,13,17\}
\end{array}.
\]}
Thus we obtain that the arrangement
\[\cala(31):=\cala(19)\cup\{\ell_{20},\ldots, 
\ell_{31}\}\] is 
inductively free with the exponents ${\rm exp}(\cala(31)) = \{ 1,13,17\},$
which completes the proof.
\end{proof}

\section{Inductively free arrangements and a 
counterexample to the containment problem.}
\label{sec:main}

The line arrangement $\cala(31)$ from the previous part of our paper turns out to be very important in the context of an open problem in the so-called containment problems for symbolic powers of homogeneous ideals. Let us recall here the definition of symbolic powers.

\begin{definition}
 Let $I \subseteq S$
 be a homogeneous ideal. For a fixed positive
integer $m$, we define the $m$-th symbolic power of the ideal $I$ as
\begin{displaymath}
I^{(m)}= S \cap \Big( \bigcap_{Q \in
\text{Ass}(I)} I^{m}_{Q} \Big),
\end{displaymath}
where $Ass(I)$ denotes the set of all prime ideals associated with I and $I_Q$
denotes the location of $I$ at $Q$.
\end{definition}
Reader unfamiliar with symbolic power and containment problem is referred to \cite{SzeSzp}.

In \cite{DraSec}, Drabkin and Seceleanu study arrangements of hyperplanes that come from irreducible complex reflection groups, proving that for some cases we have the failure of the containment
\begin{equation}
\label{fail}
(J(\cala))^{(3)} \not\subseteq (J(\cala))^2,
\end{equation}
where $J(\mathcal{A})$ denotes the radical ideal associated with the configuration of all intersection points of a given arrangement $\mathcal{A}$ and $(J(\cala))^{(3)}$ denotes the third symbolic power of $J(\cala)$.
In the light of the results obtained in \cite{DraSec}, it is natural to ask the following question.

\begin{question}\label{pyt:SD_2}
(\cite[Question 6.7.]{DraSec})
Are the containments 
$(J(\cala))^{(2r-1)}\subseteq 
(J(\cala))^r$ always satisfied 
for any $r \geq 2$ and
any hyperplane arrangement that is 
inductively free?
\end{question}
Here we answer \textbf{negative} to this question for $r=2$ and by taking the whole singular locus of the arrangement $\mathcal{A}(31)$ described in the previous section. In fact, we will show even more, namely we will present two inductively free arrangements of $31$ lines such that for one arrangement the condition \eqref{fail} holds, but for the second one it does not. In order to do so, let us present briefly a construction of the second arrangement of $31$ lines. Surprisingly, this is a simplicial line arrangement which is denoted by $\cala(31,3)$ in Grünbaum's catalogue \cite{Gru}. Here are the details.

Consider ten lines given by:
\[x\pm aez,\:2y\pm az,\:2x\pm bz,\]
where by $e=\sqrt{3}$, $a\in\{0,1\}$, 
and $b\in\{0,1,3\}$. The visualization in the affine part of the projective plane of these ten lines is presented in 
Figure~\ref{fig:tenlines}.
\begin{figure}[H]
	\centering
	\begin{tikzpicture}
	[line cap=round,line join=round,x=1.0cm,y=1.0cm, scale=0.7]
	\clip(-6.9,-6) rectangle (6,5);
	\draw [line width=1.pt] (0.8660254037844386,-4) -- (0.8660254037844386,4.8);
	\draw [line width=1.pt] (1.7320508075688772,-5) -- (1.7320508075688772,3.792672179572946);
	\draw [line width=1.pt] (3.4641016151377544,-4) -- (3.4641016151377544,4.8);
	\draw [line width=1.pt] (-0.8660254037844386,-4) -- (-0.8660254037844386,4.8);
	\draw [line width=1.pt] (0.,-5) -- (0.,3.792672179572946);
	\draw [line width=1.pt] (-1.7320508075688772,-5) -- (-1.7320508075688772,3.792672179572946);
	\draw [line width=1.pt] (-3.4641016151377544,-4) -- (-3.4641016151377544,4.8);
	\draw [line width=1.pt,domain=-10:11.621005370643537] plot(\x,{(--1.-0.*\x)/1.});
	\draw [line width=1.pt,domain=-10:11.621005370643537] plot(\x,{(-0.-0.*\x)/1.});
	\draw [line width=1.pt,domain=-10:11.621005370643537] plot(\x,{(-1.-0.*\x)/1.});
	\begin{scriptsize}
	\draw[color=black] (0,4.4) node {$x=0$};
	\draw[color=black] (-0.9,-4.5) node {$2x=-ez$};
	\draw[color=black] (0.95,-4.5) node {$2x=ez$};
	\draw[color=black] (-1.9,4.4) node {$x=-ez$};
	\draw[color=black] (1.95,4.4) node {$x=ez$};
	\draw[color=black] (-3.6,-4.5) node {$2x=-3ez$};
	\draw[color=black] (3.6,-4.5) node {$2x=3ez$};
	\draw[color=black] (-5.9,-0.3) node {$2y=0$};
	\draw[color=black] (-5.9,0.7) node {$2y=z$};
	\draw[color=black] (-5.8,-1.3) node {$2y=-z$};
	\end{scriptsize}
	\end{tikzpicture}
	\caption{. A set of $10$ 
 lines inducing 
 realizations of $\cala(31,3)$.
 }
	\label{fig:tenlines}
\end{figure}
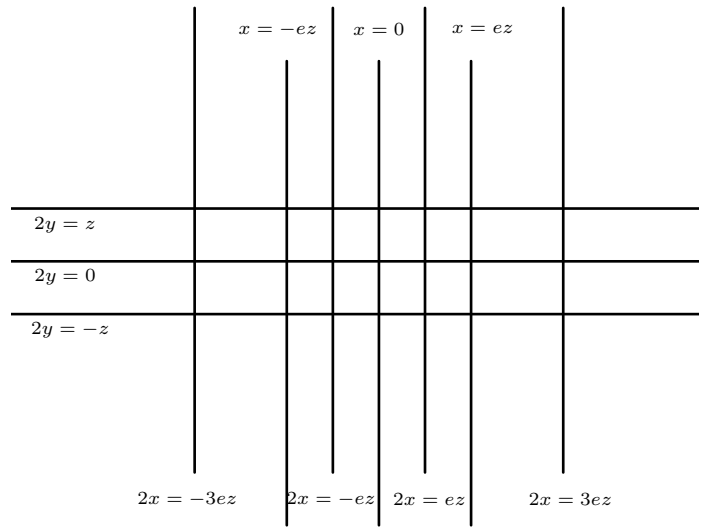
Now we perform two rotation of these lines (affinely), firstly by angle $60^\circ$, then by $120^\circ$, around the point $(0,0)$. In this way we obtain $30$ lines. Finally, to obtain $31$ lines, we take to our arrangement the line at infinity $z=0$. The resulting configuration is shown in Figure~\ref{fig:A31_3}. 

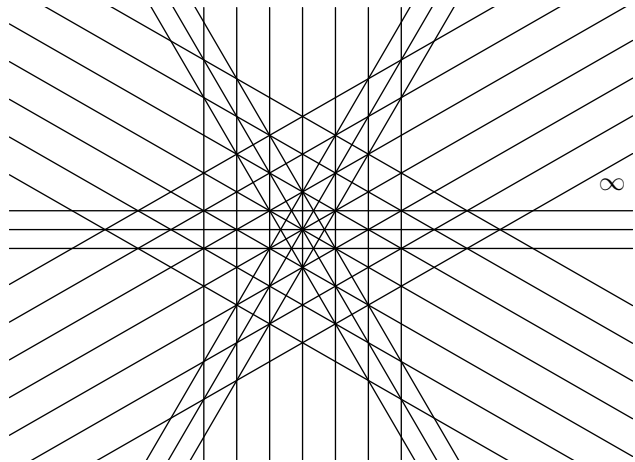
\begin{figure}[!ht]
\begin{center}
\begin{tikzpicture}[line cap=round,line join=round,x=1.0cm,y=1.0cm,scale=0.5]
\clip(-7.716171428571437,-6.1048) rectangle (8.68681904761903,5.8868761904761895);
\draw [line width=0.5pt,domain=-7.716171428571437:8.68681904761903] plot(\x,{(-0.-1.7320508075688772*\x)/1.});
\draw [line width=0.5pt,domain=-7.716171428571437:8.68681904761903] plot(\x,{(-0.-1.7320508075688772*\x)/-1.});
\draw [line width=0.5pt,domain=-7.716171428571437:8.68681904761903] plot(\x,{(-0.-0.*\x)/1.});
\draw [line width=0.5pt] (0.,-6.1048) -- (0.,5.8868761904761895);
\draw [line width=0.5pt,domain=-7.716171428571437:8.68681904761903] plot(\x,{(-0.-1.*\x)/1.7320508075688772});
\draw [line width=0.5pt,domain=-7.716171428571437:8.68681904761903] plot(\x,{(-0.-1.*\x)/-1.7320508075688772});
\draw [line width=0.5pt] (0.8660254037844386,-6.1048) -- (0.8660254037844386,5.8868761904761895);
\draw [line width=0.5pt] (-0.8660254037844386,-6.1048) -- (-0.8660254037844386,5.8868761904761895);
\draw [line width=0.5pt,domain=-7.716171428571437:8.68681904761903] plot(\x,{(-1.7320508075688772-1.*\x)/1.7320508075688772});
\draw [line width=0.5pt,domain=-7.716171428571437:8.68681904761903] plot(\x,{(--1.7320508075688772-1.*\x)/-1.7320508075688772});
\draw [line width=0.5pt,domain=-7.716171428571437:8.68681904761903] plot(\x,{(--1.7320508075688772-1.*\x)/1.7320508075688772});
\draw [line width=0.5pt,domain=-7.716171428571437:8.68681904761903] plot(\x,{(-1.7320508075688772-1.*\x)/-1.7320508075688772});
\draw [line width=0.5pt] (1.7320508075688772,-6.1048) -- (1.7320508075688772,5.8868761904761895);
\draw [line width=0.5pt] (-1.7320508075688772,-6.1048) -- (-1.7320508075688772,5.8868761904761895);
\draw [line width=0.5pt,domain=-7.716171428571437:8.68681904761903] plot(\x,{(-3.4641016151377544-1.*\x)/1.7320508075688772});
\draw [line width=0.5pt,domain=-7.716171428571437:8.68681904761903] plot(\x,{(-3.4641016151377544-1.*\x)/-1.7320508075688772});
\draw [line width=0.5pt,domain=-7.716171428571437:8.68681904761903] plot(\x,{(--3.4641016151377544-1.*\x)/-1.7320508075688772});
\draw [line width=0.5pt,domain=-7.716171428571437:8.68681904761903] plot(\x,{(--3.4641016151377544-1.*\x)/1.7320508075688772});
\draw [line width=0.5pt,domain=-7.716171428571437:8.68681904761903] plot(\x,{(--1.-1.7320508075688772*\x)/1.});
\draw [line width=0.5pt,domain=-7.716171428571437:8.68681904761903] plot(\x,{(-1.-1.7320508075688772*\x)/1.});
\draw [line width=0.5pt,domain=-7.716171428571437:8.68681904761903] plot(\x,{(-1.-1.7320508075688772*\x)/-1.});
\draw [line width=0.5pt,domain=-7.716171428571437:8.68681904761903] plot(\x,{(--1.-1.7320508075688772*\x)/-1.});
\draw [line width=0.5pt,domain=-7.716171428571437:8.68681904761903] plot(\x,{(--1.-0.*\x)/2.});
\draw [line width=0.5pt,domain=-7.716171428571437:8.68681904761903] plot(\x,{(-1.-0.*\x)/2.});
\draw [line width=0.5pt,domain=-7.716171428571437:8.68681904761903] plot(\x,{(-5.196152422706632-1.*\x)/1.7320508075688772});
\draw [line width=0.5pt,domain=-7.716171428571437:8.68681904761903] plot(\x,{(--5.196152422706632-1.*\x)/1.7320508075688772});
\draw [line width=0.5pt,domain=-7.716171428571437:8.68681904761903] plot(\x,{(-5.196152422706632-1.*\x)/-1.7320508075688772});
\draw [line width=0.5pt,domain=-7.716171428571437:8.68681904761903] plot(\x,{(--5.196152422706632-1.*\x)/-1.7320508075688772});
\draw [line width=0.5pt] (-2.598076211353316,-6.1048) -- (-2.598076211353316,5.8868761904761895);
\draw [line width=0.5pt] (2.598076211353316,-6.1048) -- (2.598076211353316,5.8868761904761895);
\draw (7.545961904761889,1.6037333333333353) node[anchor=north west] {$\infty$};
\end{tikzpicture}
\end{center}
\caption{. 
An affine realization of the arrangement 
$\cala(31,3)$. The symbol $\infty$ 
denotes the line at infinity $z=0$.
}
	\label{fig:A31_3}
\end{figure}
The following statement shows 
that the answer to Question 
\ref{pyt:SD_2} is negative.

\begin{theorem}
\label{thm:A313IFcounter}
The arrangement $\cala(31,3)$ is inductively free and one has
$(J(\cala(31,3)))^{(3)}\not \subseteq 
(J(\cala(31,3)))^2$.
\end{theorem}

\begin{proof}
To prove that the configuration 
$\cala(31,3)$ is inductively 
free, we will create a table 
containing the exponents of  
${\rm exp}(\cala^{'})$, 
the equations of lines that we add to the arrangement
$\cala(19,1)$, and then the 
exponents of ${\rm exp}(\cala^{''})$.

{\renewcommand{\arraystretch}{1.0}\begin{table}[h]
\begin{center}
\begin{tabular}{lll|lll}\\
\hline 
$\exp\:\cala^{'}$ &$\ell_i$ &  $\exp\:\cala^{''}$ &
$\exp\:\cala^{'}$ & $\ell_i$ &  $\exp\:\cala^{''}$
\\ \hline 
$\{1,8,11\}$ & $\ell_{20}: ex+y+z$ & $\{1,11\}$ & $\{1,12,13\}$ & $\ell_{26}: 4x+ez$ & $\{1,13\}$\\
$\{1,9,11\}$ & $\ell_{21}: ex+y-z$ & $\{1,11\}$ & $\{1,13,13\}$ & $\ell_{27}: 4x-ez$ & $\{1,13\}$ \\
$\{1,10,11\}$ & $\ell_{22}: ex-y+z$ & $\{1,11\}$ & $\{1,13,14\}$ & $\ell_{28}: 2x-2ey+ez$ & $\{1,13\}$\\
$\{1,11,11\}$ & $\ell_{23}: ex-y-z$ & $\{1,11\}$ & $\{1,13,15\}$ & $\ell_{29}: 2x-2ey-ez$ & $\{1,13\}$ \\
$\{1,11,12\}$ & $\ell_{24}: 2y-z$ & $\{1,11\}$ & $\{1,13,16\}$ & $\ell_{30}: 2x+2ey-ez$ & $\{1,13\}$\\
$\{1,11,13\}$ & $\ell_{25}: 2y+z$ & $\{1,13\}$  & $\{1,13,17\}$ & $\ell_{31}: 2x+2ey+ez$ & $\{1,13\}$\\
\hline
\end{tabular}
 \caption{.
 Transition from $\cala(19,1)$ to $\cala(31,3)$.
 }\label{tab:31Free}
\end{center}
\end{table}
}
Observe that each row in Table 
\ref{tab:31Free} allows us to verify condition (2) in
Definition \ref{def:IF}. Evidence of noncontainment  $(J(\cala(31,3)))^{(3)}\not\subseteq
(J(\cala(31,3)))^2$ has been provided in \cite{JLM} and due to this reason we refer to this paper for details. The verification was performed using {\tt Singular}. The affine part ($z=1$) of the element 
$F\in (J(\cala(31,3)))^{(3)}\setminus
(J(\cala(31,3)))^2$ is shown in 
Figure \ref{fig:Element}.
\end{proof}
\begin{figure}[htp]
    \centering
    \includegraphics[scale=0.6]{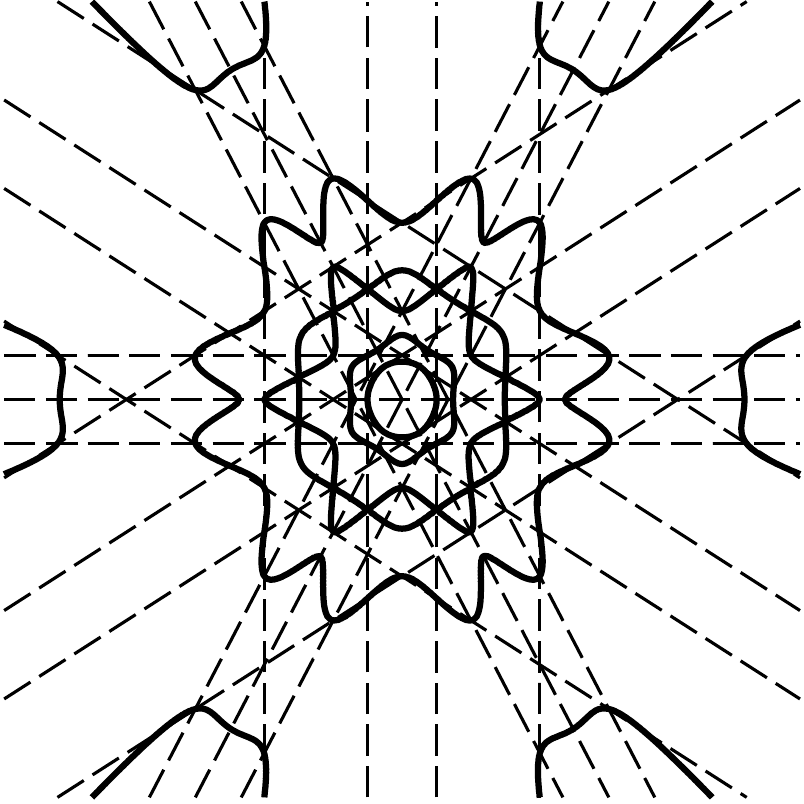}
    \caption{. The element 
    $F\in(J(\cala(31,3)))^{(3)}\setminus
(J(\cala(31,3)))^2$ consists of $21$
lines and~a~curve 
of  degree 12.
}
    \label{fig:Element}
\end{figure}

The next example illustrates the fact that being inductively free for an arrangement $\cala$ does not directly transfer into lack of containment $(J(\cala))^{(3)} \subseteq (J(\cala))^2$. 

\begin{theorem}\label{Tw:CompA31}
There are two non-isomorphic inductively free simplicial arrangements consisting of $31$ lines such that they have the same weak combinatorics, and having the property that for one arrangement the containment $(J(\cala))^{(3)}\subseteq (J(\cala))^2$ holds, not for the other. 
\end{theorem}

In other words, the weak combinatorics of line arrangements
does not determine the property of being an example for the non-containment. For the clarity of the exposition, let us recall that for an arrangement $\mathcal{A}$ of $d$ lines in the plane, the weak combinatorics is the vector of the form $(d; t_{2}, ..., t_{d})$ with $t_{i}$ being the number of $i$-fold intersection points in $\mathcal{A}$.

It is worth noticing that Theorem \ref{Tw:CompA31} should be compared with a result from \cite{FarKabBacTut2019}, where the authors observed a similar phenomenon, but in a different setting, namely in the case of real line arrangements possessing the maximal possible number of triple intersection points. 

Let us present our proof of Theorem \ref{Tw:CompA31}.
\begin{proof}
The arrangement $\cala(31,3)$ 
is constructed from configuration 
$\cala(19,1)$ by adding $12$ appropriately chosen lines,  
according to Table \ref{tab:31Free}. The 
arrangements $\cala(31,2)$ and 
$\cala(31,3)$ are not isomorphic 
simplicial arrangements of lines, and this fact is proved in \cite{Gru}, even though they have the same weak combinatorics, namely:
\[t_2=54,\:t_3=42,\:t_4=21,
\:t_5=6,\:t_6=1,\:t_8=3, \]
and $t_i=0$ for others.

Let us pass to the containment question $(J(\cala))^{(3)}\subseteq 
(J(\cala))^2$. In Theorem \ref{thm:A313IFcounter}, we explained the non-containment for the singular locus of $\mathcal{A}(31,3)$, and this check was done using \verb}SINGULAR}. In the case of $\mathcal{A}(31,2)$, we can preform exactly the same computations showing that the containment
 $$(J(\cala(31,2))^{(3)}\subseteq (J(\cala(31,2))^2$$
 does hold, which completes the proof.
\end{proof}
In Appendix, you can find a script that can be run in \verb}SINGULAR} which allows us to verify the containment $(J(\cala(31,2))^{(3)}\subseteq (J(\cala(31,2))^2$.

\begin{remark}
Let us point our here that there is a way to extend the class of inductively free arrangements $\calaif$, namely we can add the following condition (\cite[Definition 4.60 (3)]{OrlTer92}):
\begin{quote}
if there exists $H\in\cala$ such that $\cala''\in\calaif $, 
		$\cala\in\calaif$, and ${\rm exp}(\cala'')\subset {\rm exp}(\cala)$, then~ $\cala'\in\calaif$,    
\end{quote} 
then we come to the class of recursively free hyperplanes
arrangements $\calarf$. 

It is known that we have the following relations (see \cite{CuntzHoge,terao80,Zigler}) 
$$\text{inductively free } \subsetneq \text{ recursively free } \subsetneq \text{ free}.$$ 
It turns out that our example of a pair of 
arrangements $\cala(31,2)$ and $\cala(31,3)$ allows us to answer a question posed by Drabkin and Seceleanu in the negative (\cite[Question 6.8]{DraSec}). More precisely, our example shows that the containment $J(\mathcal{A})^{(3)} \subset J(\mathcal{A})^{2}$ does not hold for recursively free line arrangements $\mathcal{A}$. It is still an open question if we can find a configuration of lines which is recursively free, but not inductively free and gives negative answer.
\end{remark}

\paragraph*{Acknowledgments.}
I would like to thank Grzegorz Malara and 
Piotr Pokora for all their help, valuable 
comments and inspiring discussions.
\paragraph*{Appendix.}

\hspace*{10ex}
\hrule
\begin{verbatim}
proc PtsIdeal(poly p, poly q, poly r) {
	matrix M[2][3]=p,q,r,x,y,z;
	ideal @I=minor(M,2);
	return(std(@I));
}
option(redSB);
ring R=(0,e),(x,y,z),dp;
minpoly=e2-3;

/* The list L contains the coordinates of the singular points 
of the arrangement A(31,2). */

"loading arrangement A(31,2)...";
list L=
(-7/2e),-1/2,1,(7/2e),-1/2,1,(7/2e),1/2,1,(-7/2e),1/2,1,
(-3/2e),-11/2,1,(2e),5,1,(3/2e),11/2,1,(-2e),-5,1,
(3/2e),-11/2,1,(-2e),5,1,(-3/2e),11/2,1,(2e),-5,1,
(3/4e),5/4,1,(1/4e),7/4,1,(1/4e),-7/4,1,(3/4e),-5/4,1,
(-3/4e),5/4,1,(-1/4e),7/4,1,(-1/4e),-7/4,1,(-3/4e),-5/4,1,
(-e),-1/2,1,(-e),1/2,1,(e),-1/2,1,(e),1/2,1,
(5/2e),-1/2,1,(-3/2e),7/2,1,(5/2e),1/2,1,(-3/2e),-7/2,1,
(-5/2e),-1/2,1,(3/2e),7/2,1,(-5/2e),1/2,1,(3/2e),-7/2,1,
(-e),-4,1,(-e),4,1,(e),4,1,(e),-4,1,
(3/2e),5/2,1,(-2e),-1,1,(-3/2e),5/2,1,(2e),-1,1,
(-3/2e),-5/2,1,(2e),1,1,(3/2e),-5/2,1,(-2e),1,1,
(-1/2e),-7/2,1,(-1/2e),7/2,1,(1/2e),7/2,1,(1/2e),-7/2,1,
(-3/2e),-1/2,1,(e),2,1,(3/2e),-1/2,1,(-e),2,1,
(3/2e),1/2,1,(-e),-2,1,(-3/2e),1/2,1,(e),-2,1,
(-1/2e),5/2,1,(-1/2e),-5/2,1,(1/2e),5/2,1,(1/2e),-5/2,1,
(3/2e),3/2,1,(-3/2e),-3/2,1,(3/2e),-3/2,1,(-3/2e),3/2,1,
0,3,1,0,-3,1,(-1/4e),-1/4,1,(1/4e),1/4,1,
(1/4e),-1/4,1,(-1/4e),1/4,1,0,-1/2,1,0,1/2,1,
(-e),-1,1,(e),1,1,(-e),1,1,(e),-1,1,
0,2,1,0,-2,1,(-1/2e),-1/2,1,(1/2e),1/2,1,
(-1/2e),1/2,1,(1/2e),-1/2,1,0,1,1,0,-1,1,
(3/2e),0,1,(-3/2e),0,1,(3/4e),9/4,1,(-3/4e),-9/4,1,
(3/4e),-9/4,1,(-3/4e),9/4,1,(3e),0,1,(-3e),0,1,
(-3/2e),-9/2,1,(3/2e),9/2,1,(-3/2e),9/2,1,(3/2e),-9/2,1,
(-1/3e),0,1,(1/3e),0,1,(-1/6e),-1/2,1,(1/6e),1/2,1,
(1/6e),-1/2,1,(-1/6e),1/2,1,(2e),0,1,(-2e),0,1,
(-e),-3,1,(e),3,1,(-e),3,1,(e),-3,1,
(-1/2e),0,1,(1/2e),0,1,(1/4e),3/4,1,(-1/4e),-3/4,1,
(-1/4e),3/4,1,(1/4e),-3/4,1,(e),0,1,(-e),0,1,
(-1/2e),-3/2,1,(1/2e),3/2,1,(-1/2e),3/2,1,(1/2e),-3/2,1,
0,0,1,(e),1,0,(-e),1,0,0,1,0,-1,0,0,1,(e),0,-1,(e),0;

"generating ideals I^(3) and I^2...";
ideal I=1; ideal I3=1;
for(int i=1;i<=(size(L) div 3);i++){
	I=intersect(I,PtsIdeal(L[3*i-2],L[3*i-1],L[3*i]));
	I3=intersect(I3,(PtsIdeal(L[3*i-2],L[3*i-1],L[3*i]))^3);
	if((i mod 10) == 0){ string(i)+" points of "
        +string(size(L) div 3)+" in total used";}
}
I=std(I^2); I3=std(I3);

"number of generators of I^(3) not in I^2: "+string(size(NF(I3,I)));
\end{verbatim}
\hrule


\vskip 0.5 cm
\bigskip
Marek Janasz,
Department of Mathematics,
Pedagogical University of Krakow,\\
Podchor\c a\.zych 2,
PL-30-084 Krak\'ow, Poland. \\
\nopagebreak
\textit{E-mail address:} \texttt{marek.janasz@up.krakow.pl}
\end{document}